\documentclass{amsart}
\usepackage[
  left=1.0in,    
  right=1.0in,
  top=1.0in,
  bottom=1.0in,
  includefoot    
]{geometry}
\usepackage{amsthm,amsfonts,amsmath,amssymb}
\usepackage[colorlinks=true]{hyperref}
\usepackage[backrefs,msc-links,nobysame]{amsrefs}
\usepackage{cleveref}
\usepackage{tikz-cd}
\usepackage{tikz}
\usetikzlibrary{cd}
\usetikzlibrary[arrows.meta]
\usetikzlibrary{decorations.markings}
\tikzset{degil/.style={
            decoration={markings,
            mark= at position 0.5 with {
                  \node[transform shape] (tempnode) {$\backslash$};
                  }
              },
              postaction={decorate}
}
}

\usepackage{lineno}

\DeclareMathOperator{\lann}{l.ann}

\DeclareMathOperator{\im}{Im}

\DeclareMathOperator{\morp}{End}

\DeclareMathOperator{\id}{Id}

\DeclareMathOperator{\soc}{soc}
\DeclareMathOperator{\pur}{Pur}
\DeclareMathOperator{\udim}{u.dim}
\DeclareMathOperator{\cent}{Cent}

\theoremstyle{plain}
\newtheorem{theorem}{Theorem}[section] 
\newtheorem{lemma}[theorem]{Lemma}
\newtheorem{proposition}[theorem]{Proposition}
\newtheorem{corollary}[theorem]{Corollary}
\newtheorem{question}[theorem]{Question}

\theoremstyle{definition}
\newtheorem{definition}[theorem]{Definition}
\newtheorem{example}[theorem]{Example}

\theoremstyle{remark}
\newtheorem{remark}[theorem]{Remark} 
\newtheorem{note}[theorem]{Note}

\begin{document}


\pagenumbering{arabic}

\title{On Modules Whose Pure Submodules Are Essential in Direct Summands}

\author[Gupta]{Kaushal Gupta}
\address{(Gupta) Applied Sciences Cluster, School of Advanced Engineering, University of Petroleum and Energy Studies, Dehradun, Uttarakhand, India-248007}
\email{kaushal.gupta@ddn.upes.ac.in}

\author[Gera]{Theophilus Gera} 
\address{(Gera) Department of Mathematics, Sardar Vallabhbhai National Institute of Technology, Surat, Gujarat, India-395007}
\email{geratheophilus@gmail.com}
\thanks{The second author would like to thank the Department of Education, Government of India, for the financial assistance.}

\author[Sharma]{Amit Sharma} 
\address{(Sharma) Department of Mathematics, Sardar Vallabhbhai National Institute of Technology, Surat, Gujarat, India-395007}
\email{amitsharma@amhd.svnit.ac.in}

\author[Gupta]{Ashok Ji Gupta}
\address{(Gupta) Department of Mathematical Sciences, Indian Institute of Technology (B.H.U.), Varanasi, Uttar Pradesh, India-221005}
\email{agupta.apm@itbhu.ac.in}

\date{\today}

\begin{abstract}
    We introduce the notion of \emph{pure extending modules}, a refinement of classical extending modules in which only pure submodules are required to be essential in direct summands. Fundamental properties and characterizations are established, showing that pure extending and extending modules coincide over von Neumann regular rings. As an application, we prove that pure extending modules admit decomposition patterns analogous to those in the classical theory, including a generalization of the Osofsky--Smith theorem: a cyclic module whose proper factor modules are pure extending decomposes into a finite direct sum of pure-uniform submodules. We establish sufficient conditions under which central quasi-morphicity and central morphicity coincide, notably for finitely generated, nonsingular, pure extending modules over noetherian rings, and for modules over semisimple artinian rings.
\end{abstract}



\maketitle


\setcounter{tocdepth}{3}

\tableofcontents

\section{Introduction}

The classical notion of extending modules, where every submodule is essential in a direct summand, plays a foundational role in the decomposition theory of modules. Such modules satisfy the $C_1$ condition, a structural strengthening of summand-closure that has been widely studied in relation to torsion theories, purity, and homological algebra. However, in many algebraic settings governed by homological purity, particularly those involving flatness, divisibility, or exactness under tensor, this condition proves overly rigid. Motivated by this, we propose a conceptual refinement: requiring only that \emph{pure submodules} be essential in direct summands.

This shift captures how, for pure extending modules, the focus moves from submodule inclusion to the homological traceability of pure exact sequences. This perspective motivates the study of pure extending modules as a homological analogue of extending modules within the framework of pure-exact decompositions. We introduce and study the class of \emph{pure extending modules}, in which pure submodules (i.e., those that preserve exactness under tensor products or reflect ideal-wise divisibility) are required to be essential in direct summands. This perspective reorients the $C_1$ condition in terms of exactness behavior, offering a natural and flexible framework for analyzing decomposition in both classical and relative module-theoretic settings.

The study of essential embeddings and summand structures has deep roots, beginning with von Neumann's foundational work on continuous geometries and Utumi's analysis of rings where every left ideal is essentially contained in a projective summand \cite{utumi1965continuous}. These ideas were extended to modules by Jeremy \cite{jeremy1971sur} and further developed through various generalizations. Independently, Chatters and Hajarnavis introduced CS-modules (``complements are summands'') in \cite{chatters1977rings}, while Harada coined the term ``extending module,'' which later became central in modern module theory. M\"{u}ller systematized these ideas through the now-standard $C_1$, $C_2$, and $C_3$ conditions \cite{mohamed1990continuous}. The present work builds on this lineage by proposing a structurally natural weakening of the $C_1$ condition grounded in purity.

To formalize our main definition: a module $M$ is said to be \emph{pure extending} if every pure submodule of $M$ is essential in a direct summand. This condition weakens the classical requirement that \emph{all} submodules be essential in direct summands, yet it preserves much of the decomposability behavior central to extending theory. In particular, pure extending modules retain strong control over pure submodules and factor modules, interact naturally with torsion and flatness conditions, and provide a framework for analyzing decompositions governed by purity rather than full essentiality.

We position pure extending modules within a lattice of injectivity-related conditions. A module is \emph{pure-injective} if all pure extensions split, and \emph{quasi-pure-injective} if it is pure-injective relative to its own submodules \cite{harmanci2015pure}. Pure extending modules weaken these conditions by replacing splitting with essentiality. The classical pattern between injective, quasi-injective, and extending modules carries over to the pure setting, yielding the following hierarchy:
\[
\begin{tikzcd}[column sep=huge]
\text{Injective} \ar[r,Rightarrow] \ar[d,Rightarrow] & \text{Quasi-injective} \ar[r,Rightarrow] \ar[d,Rightarrow] & \text{Extending} \ar[d,Rightarrow] \\
\text{Pure-injective} \ar[r,Rightarrow] & \text{Quasi-pure-injective} \ar[r,Rightarrow] & \text{Pure extending}
\end{tikzcd}
\]
None of the implications is reversible in general, as shown through examples (see Example \ref{PE=/> E} and \cite{fieldhouse1969pure}). Hence, pure extending modules form the natural bridge between decomposability via essential submodules and purity-based control over module structure.

Several classical questions concerning extending modules remain unresolved. Chief among these is understanding how decomposability behaves under direct sums and factor constructions. Specifically, one seeks to determine when the direct sum of extending modules remains extending, and whether the property of being extending descends to factor modules.

\begin{question}\label{q1.1}
\begin{enumerate}
    \item[(1)] If $M = M_1 \oplus M_2$ and each $M_i$ is extending, is $M$ necessarily extending?
    \item[(2)] If all factor modules of $M$ are extending, must $M$ decompose as a finite direct sum of uniform modules?
\end{enumerate}
\end{question}

Although finite direct sums of extending modules are not necessarily extending, Birkenmeier, M\"{u}ller, and Rizvi \cite{birkenmeier2002modules} established that such sums are always FI-extending. We contribute a partial resolution by showing that pure extending modules form a class closed under finite direct sums (see Theorem~\ref{direct sum}). Moreover, by an earlier structural characterization (see Proposition~\ref{von-Neumann regular <=> purec1 -> c1}), we prove that over a von Neumann regular ring, pure extending and extending modules coincide. Consequently, Question~\ref{q1.1}(1) admits an affirmative answer in this setting (see Corollary~\ref{vnr, closure extending}). Nevertheless, regularity is not essential to our analysis, and we construct examples both supporting and refuting closure beyond this context (see Example~\ref{extending-not-closed}).

To broaden the theory, we also introduce \emph{RD-pure extending modules}, defined using element-wise divisibility ($rP = rM \cap P$ for all $r \in R$). While every pure submodule is RD-pure, the converse fails in general. In this direction, we show that RD-pure extending modules strictly contain pure extending modules (see Example \ref{RD-pure not pure}), but under additional conditions, such as flatness over right perfect rings, the two notions coincide (see Corollary \ref{RD-pure => pure}). This makes RD-pure extending modules a flexible tool for controlling decomposability in broader settings.

The theory developed here yields two core applications: (a) decomposition theorems generalizing the Osofsky--Smith result to purity-sensitive contexts, and (b) an investigation of a recent question raised by Dehghani and Sedaghatjoo about the morphic behavior of modules under central constraints.

To start with, we generalize a decomposition problem posed by Osofsky and Smith \cite{osofsky1991cyclic}, who asked whether a cyclic module with all cyclic submodules completely extending decomposes into uniform summands. Under a purity assumption, we prove that a cyclic module whose proper factor modules are pure extending decomposes into a finite direct sum of pure-uniform modules (see Theorem \ref{cyclic fac PE => uni. submod.}). The proof involves endoartinian decomposition theorems and a finiteness condition on pure uniform submodules. We also provide partial answers to Question~\ref{q1.1}(2) and refine these under assumptions (see Corollary \ref{coropen1}).

Finally, we address a recent question posed by Dehghani and Sedaghatjoo \cite{dehghani2025centrally} concerning whether every centrally quasi-morphic module must be centrally morphic. We establish sufficient conditions under which the equivalence holds, notably for finitely generated, nonsingular, pure-extending modules over noetherian rings (see Theorems~\ref{main1}--\ref{main2}, Proposition~\ref{propopen}, and Corollary~\ref{coropen}). We also show that centrally morphic modules need not be Rickart (see Example~\ref{ex:cm_not_rickart}). The general question of whether every centrally quasi-morphic module is centrally morphic remains open.

Throughout, all rings are associative with identity and all modules are unital right modules unless stated otherwise. For a module $M$, we write $\morp_R(M)$ (or $S$) for the endomorphism ring, acting on the left. We use $\leq$, $\leq^\oplus$, and $\leq_e$ for submodule inclusion, direct summand, and essential submodule, respectively. A module is \emph{nonsingular} if its singular submodule vanishes, and \emph{uniform} if all nonzero submodules intersect nontrivially. Standard references include \citelist{\cite{lam2001firstcourse} \cite{lam1999lectures} \cite{wisbauer1991foundations}}. 

\section{Properties and Characterizations} \label{sec2}

\subsection{Pure Extending Modules}

The notion of \emph{pure extending modules} emerges from an effort to reconcile two foundational ideas in module theory: essential extensions and homological purity. While extending modules demand that every submodule be essential in a direct summand, this requirement can be overly restrictive in contexts where structural behavior is governed by purity—whether ideal-wise or element-wise—rather than arbitrary submodules. Focusing on pure submodules offers a more nuanced and robust framework: these submodules preserve tensor exactness and often capture divisibility conditions central to torsion-theoretic and non-flat settings. We thus define a module to be \emph{pure extending} if every pure submodule is essential in a direct summand. This class properly contains the class of extending modules and retains many of their decomposition-theoretic properties. Through examples (see Example \ref{PE=/> E}) and structural characterizations (see Propositions~\ref{summandpurec1}, \ref{von-Neumann regular <=> purec1 -> c1}), we show that pure extending modules provide a natural generalization suited to addressing longstanding questions in decomposition theory and morphic module classification.

\begin{definition} \label{def PE}
    A module $M$ is said to be \emph{pure extending} if every pure submodule of $M$ is essential in a direct summand of $M$.
\end{definition}

\begin{example} \label{PE=/> E}
\begin{enumerate}
    \item Let  \( M = \mathbb{Z}_2 \oplus \mathbb{Z}_8 \) as a \( \mathbb{Z} \)-module. Since $M$ is finitely generated over the noetherian ring $\mathbb{Z}$, every pure submodule is a direct summand by \cite{lam1999lectures}*{Corollary 4.91}, so $M$ is pure extending. The submodule $N=<(\bar{1},\bar{2})>$ is cyclic of order 2 and projects nontrivially onto the $\mathbb{Z}_2$-summand while its projection to the $\mathbb{Z}_8$-summand is not essential; checking the possible direct summands $A\oplus B$ of $M$ shows $N$ is not essential in any of them, so $M$ fails to be extending. 
    \item The right $R$-module $M=\begin{pmatrix}
        \mathbb{Z} & \mathbb{Z} \\
        0 & 0
    \end{pmatrix}$, where $R=\begin{pmatrix}
        \mathbb{Z} & \mathbb{Z} \\
        0 & \mathbb{Z}
    \end{pmatrix}$, is pure extending but not extending. It is not extending because the submodule $M=\begin{pmatrix}
        0 & \mathbb{Z} \\
        0 & 0
    \end{pmatrix}$ is not essential in any direct summand. It is pure extending because its only pure submodules are $0$ and $M$ itself, which are trivially essential in a direct summand. Moreover, $R$ is not right extending as shown in \cite{chatters1977rings}*{Example 6.2}.
\end{enumerate} \qed
\end{example}

These examples confirm that the class of pure extending modules strictly contains the class of extending modules. We now establish some basic closure properties of pure extending modules, beginning with their behaviour under direct summands, which mirrors the classical case.

\begin{proposition} \label{summandpurec1}
Let $M$ be a pure extending module. Then every direct summand of $M$ is pure extending.
\end{proposition}

\begin{proof}
Write $M=N\oplus N'$, and let $P\leq N$ be pure in $N$. Since the inclusion $N\hookrightarrow M$ is split (hence pure), $P$ is pure in $M$. As $M$ is pure extending there exists a direct summand $D\le M$ with $P\leq_e D$. Write $M=D\oplus C$. Then
\[
N=(D\oplus C)\cap N=(D\cap N)+(C\cap N),
\]
and $(D\cap N)\cap(C\cap N)=(D\cap C)\cap N=0$, so $N=(D\cap N)\oplus(C\cap N)$; therefore $D\cap N$ is a direct summand of $N$.

Finally, let $0\neq x\in D\cap N$. Since $P$ is essential in $D$ there is $r\in R$ with $0\neq r x\in P$. Hence every nonzero submodule of $D\cap N$ meets $P$, i.e. $P\leq_e(D\cap N)$ inside $N$. Thus $N$ is pure extending.
\end{proof}

It is well known that every direct summand is a pure submodule, but the converse does not hold in general. For a simple concrete example, let $R=\bigcup_{n\ge1} \Bbbk[[x^{1/n}]]$ be the valuation domain of Puiseux series over a field \(k\).  This is a (non-noetherian) valuation (hence Pr\"{u}fer) domain, so every ideal is flat and therefore pure as an \(R\)-submodule.  Since \(R\) is not noetherian, it admits non-principal ideals \(I\), and any such non-principal ideal is not projective and hence not a direct summand of \(R\).  Thus \(I\) is a pure submodule of \(R\) which is not a direct summand. A ring $R$ is called a \emph{pure direct summand} (PDS) ring if every pure submodule of an $R$-module is a direct summand (see \cite{fieldhouse1969pure}).

\begin{proposition} \label{puresubmod}
    If $R$ is a PDS ring, then every pure submodule of a pure extending module is pure extending.
\end{proposition}

\begin{proof}
Let $M$ be a pure extending module and $P\leq M$ a pure submodule. Since $R$ is PDS, $P$ is a direct summand of $M$. Hence, $P$ inherits the pure extending property from $M$ by Proposition \ref{summandpurec1}.
\end{proof}

\begin{example} \label{submod.pure.ext} 
Fix a prime $p$. Let $A=\bigoplus_{n\ge 1}\mathbb{Z}/p^n\mathbb{Z}$. It is standard that $A$ has a pure subgroup $U$ which is not a direct summand. Hence $A$ is not pure extending. Let $M=E(A)$ be the injective hull of $A$. Then $M$ is divisible, hence injective. Every injective module is pure-injective, so every pure submodule of $M$ splits and is a direct summand; therefore, $M$ is pure extending. Thus, $A$ embeds in a pure extending module $M$, while failing to be a pure extending module itself, exhibiting non-heredity. \qed
\end{example}

This shows that, even when the module is injective (hence pure extending), submodules need not be pure extending. Since the above construction takes place over \( \mathbb{Z} \), this also shows that \( \mathbb{Z} \) is not a PDS ring, as it admits a module with a pure submodule that is not a direct summand.

We next identify several natural classes of rings and modules that automatically satisfy the pure extending property, beginning with quasi-pure-injective and divisible modules. 

\begin{definition}[{\citelist{\cite{azumaya1989rings} \cite{harmanci2015pure}}}]
Let \( N \leq M \) be a pure submodule. We say that \( M \) is a \emph{pure-essential extension} of \( N \) if, for every submodule \( M' \leq M \) with \( M' \cap N = 0 \), the image of \( N \) is not pure in \( M/M' \).

A module \( E \) is called the \emph{pure-injective envelope} of \( M \) if \( E \) is pure-injective and \( M \leq E \) is a pure-essential extension.

A module \( M \) is said to be \emph{quasi-pure-injective} if every homomorphism from a pure submodule of \( M \) into \( M \) extends to all of \( M \); equivalently, \( M \) is \( M \)-pure-injective.

A module \( M \) is called \emph{pure-split} if every pure submodule of \( M \) is a direct summand.
\end{definition}

\begin{proposition} \label{prop2.8}
A module $M$ is pure extending in each of the following cases:
\begin{enumerate}
    \item \( M \) is fully invariant in its pure-injective envelope;
    \item \( M \) is quasi-pure-injective;
    \item \( M \) is a \( \mathbb{Z} \)-module that is either finitely generated or divisible;
    \item \( M \) is pure-split.
    \item \( R \) is local and \( M = R_R \);
    \item \( R \) is a PDS ring;
    \item \( M \) is flat and cotorsion.
\end{enumerate}
\end{proposition}

\begin{proof}
\begin{enumerate}
\item This is an immediate consequence of \cite{harmanci2015pure}*{Lemma 3.1}.

\item Let \( P \leq M \) be a pure submodule. Denote by \( PE(P) \) and \( PE(M) \) the pure-injective envelopes of \( P \) and \( M \), respectively. Since \( P \leq M \) is pure, \( PE(P) \) is a direct summand of \( PE(M) \) (see \cite{azumaya1989rings}). Because \( M \) is quasi-pure-injective, it is itself a direct summand of \( PE(M) \). Hence \( P \) is essential in \( M \cap PE(P) \), and \( M \cap PE(P) \) is a direct summand of \( M \). Therefore \( M \) is pure extending.

\item If \( M \) is finitely generated over the Noetherian ring \( \mathbb{Z} \), then every pure submodule of \( M \) is a direct summand by \cite{lam1999lectures}*{Corollary 4.91}. If \( M \) is divisible, then by Baer’s criterion it is injective as a \( \mathbb{Z} \)-module, hence extending, and consequently pure extending.

\item If every pure submodule of \( M \) is a direct summand, then each is trivially essential in itself, so \( M \) is pure extending.
\item By \cite{fieldhouse1970pure}*{Theorem 3}, when \( R \) is local, the only pure submodules of \( R_R \) are \( 0 \) and \( R \) itself; hence \( R_R \) is pure extending.
    \item Immediate, since over a PDS ring every pure submodule is a direct summand.
    \item If \( M \) is flat and cotorsion, then it is pure-injective. Therefore \( M \) is pure extending by (2).
\end{enumerate}
\end{proof}

The preceding conditions offer flexible criteria that guarantee the pure extending property. We now investigate the converse: when does a pure extending module necessarily satisfy the classical extending condition? This leads to a ring-theoretic characterization, revealing that the distinction between these two classes collapses precisely over von Neumann regular rings.

\begin{proposition} \label{von-Neumann regular <=> purec1 -> c1}
Let \( R \) be a von Neumann regular ring. Then a right \( R \)-module \( M \) is pure extending if and only if it is extending.
\end{proposition}

\begin{proof}
($\Rightarrow$). Suppose \( M \) is pure extending. Let \( N \leq M \) be any submodule. Since \( R \) is von Neumann regular, every right \( R \)-module is flat, so \( N \) is a pure submodule of \( M \). By the pure extending property, \( N \) is essential in a direct summand of \( M \), which is precisely the definition of \( M \) being extending.

($\Leftarrow$). Conversely, if \( M \) is extending, then for every pure submodule \( P \leq M \), the extending property ensures that \( P \) is essential in some direct summand of \( M \); hence \( M \) is pure extending.

Therefore, over a von Neumann regular ring, the classes of pure extending and extending modules coincide.
\end{proof}

The pure extending property is preserved under equivalences of module categories, as one would expect from its definability in terms of purity, essentiality, and summand structure. This invariance is formalized in the following result.

\begin{proposition}
Let \( R \) and \( S \) be Morita equivalent rings, and let 
\( \mathcal{F} \colon \mathrm{Mod}\text{-}R \to \mathrm{Mod}\text{-}S \)
be an equivalence of categories. Then an \( R \)-module \( M \) is pure extending if and only if \( \mathcal{F} (M) \) is pure extending as an \( S \)-module.
\end{proposition}

\begin{proof}
Morita equivalences preserve exact sequences, direct summands, and essential submodules (see \cite{lam1999lectures}*{\S 18}). In particular, they preserve purity and pure submodules. Hence, if \( M \) is pure extending, so is \( \mathcal{F} (M) \), and conversely.
\end{proof}

We now examine the behavior of the pure extending property under direct sums. Since purity interacts well with summands, it is natural to ask when a direct sum of pure extending modules remains pure extending. Extending the following result inductively yields a finite closure property for the class of pure extending modules.

\begin{theorem} \label{direct sum}
Let \( M = M_1 \oplus M_2 \). Then \( M \) is pure extending if and only if both \( M_1 \) and \( M_2 \) are pure extending.
\end{theorem}

\begin{proof}
($\Rightarrow$). Assume \( M \) is pure extending. Since \( M_1 \) and \( M_2 \) are direct summands of \( M \), Proposition~\ref{summandpurec1} implies that each \( M_i \) is pure extending.

($\Leftarrow$). Suppose that \( M_1 \) and \( M_2 \) are pure extending. Let \( P \leq M \) be a pure submodule, and let \(\pi_i : M \to M_i\) be the canonical projections for \(i = 1, 2\). Then \(\pi_1(P)\) is pure in \(M_1\), so there exists a direct summand \(D_1 \le M_1\) such that \(\pi_1(P) \leq_e D_1\). Similarly, there exists a direct summand \(D_2 \le M_2\) with \(\pi_2(P) \leq_e D_2\). Set \( D = D_1 \oplus D_2 \), which is a direct summand of \( M \).

For any \( (p_1, p_2) \in P \), we have \( p_i \in \pi_i(P) \leq D_i \), 
hence \( P \leq D \). Since \( D \) is a direct summand of \( M \) and purity is preserved under intersections with summands, \( P \) is pure in \( D \).

We now verify essentiality. Let \( 0 \neq x = (d_1, d_2) \in D \). If \( d_1 \neq 0 \), then, since \(\pi_1(P) \leq_e D_1\), there exists \( r \in R \) such that \( 0 \neq d_1 r \in \pi_1(P) \). Choose \( p = (d_1 r, p_2) \in P \) for some \( p_2 \in \pi_2(P) \). Consider \( x r - p = (0, d_2 r - p_2) \in D \). If \( x r - p = 0 \), then \( x r = p \in P \), so \( P \cap xR \neq 0 \). If \( x r - p \neq 0 \), then \( d_2 r - p_2 \neq 0 \); since \(\pi_2(P) \leq_e D_2\), there exists \( s \in R \) such that \( 0 \neq (d_2 r - p_2)s \in \pi_2(P) \). Then $(xr - p)s = (0, (d_2 r - p_2)s) \in 0 \oplus \pi_2(P) \leq P$, so \( x(rs) \in P \) and \( x(rs) \neq 0 \). Thus \( P \cap xR \neq 0 \). The case \( d_2 \neq 0 \) is symmetric. Hence \( P \leq_e D \), and \( M \) is pure extending.
\end{proof}

The above theorem does not extend to infinite direct sums.

\begin{example}
Let \( R = \mathbb{Z} \) and \( M_i = \mathbb{Z} \) for each \( i \in \mathbb{N} \). Set \( M = \bigoplus_{i=1}^{\infty} M_i = \bigoplus_{i=1}^{\infty} \mathbb{Z} \), and consider the submodule
\[
P = \left\{ (n_i) \in M \,\middle|\, \textstyle\sum_{i=1}^{\infty} n_i = 0 \right\},
\]
where the sum is finite since each element of \( M \) has finite support. Each \( M_i = \mathbb{Z} \) is pure extending, as \( \mathbb{Z} \) is a uniform domain. However, \( M \) itself is not pure extending, because \( P \) is a pure submodule of \( M \) that is not essential in any direct summand of \( M \). \qed
\end{example}

As a corollary, we recover a known result for classical extending modules over von Neumann regular rings.

\begin{corollary} \label{vnr, closure extending}
Let \( R \) be von Neumann regular. Then the class of extending right \( R \)-modules is closed under finite direct sums.
\end{corollary}

\begin{proof}
Over such rings, every submodule is pure, so pure extending and extending modules coincide by Proposition \ref{von-Neumann regular <=> purec1 -> c1}. 
The result then follows from Theorem \ref{direct sum}.
\end{proof}

The following examples illustrate both the scope and the limitation of this closure property.

\begin{example} \label{extending-not-closed}
\begin{enumerate}
  \item \emph{A von Neumann regular instance.}  Let \( R = \Bbbk \) be a field. Since \( \Bbbk \) is semisimple artinian, every \( \Bbbk \)-module is injective. In particular, for any finite index set \( I \), the direct sum 
  \( M = \bigoplus_{i \in I} \Bbbk \cong \Bbbk^{|I|} \) is injective and hence extending. This confirms the corollary in the classical von Neumann regular setting.

  \item \emph{A non-von Neumann regular instance.}  Let \( R = \Bbbk[x]/(x^2) \), where \( \Bbbk \) be a field. This ring is a finite-dimensional local Frobenius algebra, so \( R \) is self-injective. Consequently, \( R \oplus R \) is injective and hence extending. However, since \( x \neq 0 \) is nilpotent, \( R \) is not reduced and therefore not von Neumann regular. Thus, while von Neumann regularity guarantees closure, it is not a necessary condition.
\end{enumerate} \qed
\end{example}

We now provide several structural characterizations of pure extending modules, particularly those that reveal deeper connections with the nature of the base ring. In the first result, we characterize von Neumann regularity in terms of the flatness of pure extending modules.

\begin{proposition} \label{fe1}
Let \( R \) be a ring. The following statements are equivalent:
\begin{enumerate}
    \item \( R \) is von Neumann regular;
    \item Every pure extending right \( R \)-module is flat.
\end{enumerate}
\end{proposition}

\begin{proof}
(1) \(\Rightarrow\) (2). Over a von Neumann regular ring, every module is flat \cite{lam1999lectures}*{Theorem 4.21}, so the claim follows immediately.

(2) \(\Rightarrow\) (1). Let \( M \) be a right \( R \)-module and \( PE(M) \) its pure-injective hull. The canonical sequence
\[
0 \longrightarrow M \longrightarrow PE(M) \longrightarrow PE(M)/M \longrightarrow 0
\]
is pure exact. By hypothesis, \( PE(M) \) is pure extending and hence flat. 
Then by \cite{lam1999lectures}*{Theorem 4.86}, the quotient \( PE(M)/M \) is flat, so \( M \) itself is flat. Since every right module is flat, \( R \) is von Neumann regular.
\end{proof}

We next characterize semisimple rings in terms of several module-theoretic conditions involving purity, injectivity, and decomposition.

\begin{theorem}  \label{fe2}
Let \( R \) be a ring. The following conditions are equivalent:
\begin{enumerate}
    \item \( R \) is semisimple;
    \item Every pure \( C_3 \) \( R \)-module is projective;
    \item Every pure \( C_2 \) \( R \)-module is projective;
    \item Every quasi-pure-injective \( R \)-module is projective;
    \item Every pure-injective \( R \)-module is projective;
    \item Every pure extending \( R \)-module is projective.
\end{enumerate}
\end{theorem}

\begin{proof}
(1) \(\Rightarrow\) (2)--(6). If \( R \) is semisimple, all modules are projective and injective; hence, the listed classes are trivially projective.

(2) \(\Rightarrow\) (3). Immediate, since the class of \( C_2 \)-modules is contained in that of \( C_3 \)-modules.

(3) \(\Rightarrow\) (4). By \cite{maurya2022pure}*{Proposition 6}, every quasi-pure-injective module is pure and satisfies the \( C_2 \) condition.

(4) \(\Rightarrow\) (5). Every pure-injective module is, in particular, quasi-pure-injective.

(5) \(\Rightarrow\) (1). If all pure-injective modules are projective, then every injective module is projective; hence \( R \) is semisimple by \cite{wisbauer1991foundations}*{Proposition 20.7}.

(6) \(\Rightarrow\) (1). Since every injective module is pure-injective and every pure-injective module is pure extending, injective modules are pure extending. By (6), they are therefore projective, forcing \( R \) to be semisimple.

(1) \(\Rightarrow\) (6). Over semisimple rings, projective, injective, and extending modules coincide, and every submodule is pure.
\end{proof}

\subsection{RD-pure Extending Modules}

The notion of RD-pure extending modules isolates a weaker, element-wise form of purity. Classical (ideal-wise) purity requires $IP=IM\cap P$ for all ideals $I\subseteq R$, whereas RD-purity asks only for $rP=rM\cap P$ for all $r \in R$ \cite{wisbauer1991foundations}. Over principal ideal domains (in particular, over $\mathbb{Z}$), these notions coincide, since every ideal is principal. The distinction becomes meaningful over rings with zero divisors, where RD-purity can be strictly weaker:
\begin{center}
    Pure $\subsetneq$ RD-pure
\end{center}

For example, with $R=\Bbbk[x,y]/(x,y)^2$, $M=R\oplus R$. Consider the submodule $N=\{(\bar{x}a,\bar{y}a) \ : \ a\in R \}$ of $M$. Because $R=\Bbbk\oplus \Bbbk\bar{x}\oplus \Bbbk \bar{y}$ with $\bar{x}^2=\bar{x}.\bar{y}=\bar{y}^2=0$, it suffices to check $rN=rM\cap N$ for $r=1,\bar{x},\bar{y}$. For $r=1$ the equality is trivial; for $r=\bar{x}$ (resp. $\bar{y}$) we have $\bar{x}N=\{0\}=\bar{x}M\cap N$ (resp. $\bar{y}N=\{0\}=\bar{y}M\cap N$) because products of two radical elements are zero and $\bar{x},\bar{y}$ are $\Bbbk$-linearly independent. Hence, $N$ is RD-pure in $M$ but not pure: Consider the ideal \( I = (\bar{x},\bar{y}) \) and the element \( (\bar{x},0) \in M \). Then $(\bar{x},0) \in IM \cap N$ since $(\bar{x},0) = \bar{x}\cdot(1,0) + \bar{y}\cdot(0,0) \in IM$ and also $(\bar{x},0) = (\bar{x}\cdot 1,\ \bar{y}\cdot 0) \in N$. However, $IN = \left\{ \big(\bar{x}^2 a + \bar{x}\bar{y}b,\ \bar{x}\bar{y}a + \bar{y}^2 b\big) : a,b \in R \right\} = 0$ because in \(R\) we have the relations $\bar{x}^2 = \bar{x}\bar{y} = \bar{y}^2 = 0$. Thus, $(\bar{x},0) \notin IN$, which shows that $IM \cap N \neq IN$ and therefore \(N\) is not pure in \(M\). This shows that RD-purity can capture element-wise control in contexts where ideal-wise control is too restrictive.

A classical result of Fieldhouse \cite{fieldhouse1969pure} shows that purity and RD-purity coincide precisely for flat modules. This equivalence plays a key role in our results: in Proposition \ref{free(proj.) RD-PE <=> PE}, we show that projective modules are RD-pure extending if and only if they are pure extending. This is further generalized in Corollary \ref{RD-pure => pure} to all flat modules over right perfect rings.

Thus, the class of RD-pure extending modules offers a more flexible framework for studying summand-essential substructures, preserving divisibility behavior without requiring full purity.

\begin{definition}[\cite{wisbauer1991foundations}*{34.8(c)}]
    A submodule \( P \leq M \) is said to be \emph{relatively divisible (RD-pure)} if \( rP = rM \cap P \) for every \( r \in R \).
\end{definition}

It follows directly from the definition that every pure submodule is RD-pure. A module \( M \) is said to be \emph{RD-pure extending} if every RD-pure submodule of \( M \) is essential in a direct summand.

\begin{proposition}
    Every pure extending module is RD-pure extending.
\end{proposition}

\begin{proof}
    Indeed, since every pure submodule is RD-pure, every pure extending module remains an RD-pure extending module.
\end{proof}

As noted in \cite{lam1999lectures}*{p.~159}, RD-pure submodules need not be pure. Consequently, the class of RD-pure extending modules properly contains that of pure extending modules, as illustrated in Example \ref{RD-pure not pure}.

We next establish two closure properties for RD-pure extending modules.

\begin{proposition}
Let \( M \) be an RD-pure extending module. Then:
\begin{enumerate}
    \item Every direct summand of \( M \) is RD-pure extending.
    \item Every RD-pure submodule of \( M \) is RD-pure extending.
\end{enumerate}
Moreover, if \( R \) is a PDS ring and \( M \) is a flat pure extending module, then every RD-pure submodule of \( M \) is RD-pure extending.
\end{proposition}

\begin{proof}
\begin{enumerate}
    \item Direct summands preserve RD-pure submodules and their essentiality in summands, hence the property is inherited.
    \item Let $K$ be an RD-pure submodule of $M$, and let $P$ be an RD-pure submodule of $K$. Since the composition of RD-pure embeddings is RD-pure, $P$ is an RD-pure submodule of $M$. As $M$ is RD-pure extending, $P$ is essential in a direct summand of $M$. A standard argument (similar to the proof of Proposition \ref{summandpurec1}) shows that this property is inherited by $K$, so $K$ is RD-pure extending.
\end{enumerate}
In a PDS ring, every pure submodule is a direct summand. Since \( M \) is flat, RD-pure submodules coincide with pure submodules. The result now follows from Proposition \ref{puresubmod}.
\end{proof}

\begin{proposition} \label{free(proj.) RD-PE <=> PE}
A projective module (in particular, a free module) is RD-pure extending if and only if it is pure extending.
\end{proposition}

\begin{proof}
Since RD-purity and purity coincide for projective modules (see \cite{lam1999lectures}*{p.~159}), every RD-pure submodule of a projective module is pure, and conversely. Hence, the extending conditions are equivalent.
\end{proof}

\begin{corollary} \label{RD-pure => pure}
Let \( R \) be a right perfect ring, and let \( M \) be a flat \( R \)-module. Then \( M \) is RD-pure extending if and only if it is pure extending.
\end{corollary}

\begin{proof}
In a right perfect ring, every flat right module is projective (see \cite{lam2001firstcourse}*{Theorem 24.25}). Hence, the equivalence follows immediately from Proposition~\ref{free(proj.) RD-PE <=> PE}.
\end{proof}

The preceding framework initially provides a foundation for decomposition results involving factor modules and cyclicity conditions, extending the classical results of Osofsky and Smith. Later, we explore the relationship between morphic and pure extending modules.

\section{Applications} \label{sec4}

\subsection{A Decomposition Approach Toward the Osofsky--Smith Theorem} \label{3.1}

\subsubsection{Historical Background and Prior Results}

Osofsky and Smith established the following classical result in \cite{osofsky1991cyclic}:

\begin{theorem}[\cite{osofsky1991cyclic}]
    Let \( M \) be a cyclic module such that every cyclic submodule of \( M \) is completely extending. Then \( M \) is a finite direct sum of uniform modules.
\end{theorem}

Here, a module is said to be \emph{completely extending} if every quotient module is extending. This result refined earlier theorems by removing previously required conditions and provided a clean decomposition criterion for cyclic modules. Further generalizations were obtained by Dung \cite{dung1992generalized} and by Huynh, Dung, and Wisbauer \cite{huynh1991modules}. A natural question that remained open was whether the hypothesis on cyclic subfactors could be relaxed to arbitrary factor modules (see \cite{dung1994extending}*{p. 65}).

\begin{quote}
Let \( M \) be a cyclic (finitely generated) module with all factor modules extending. Is $M$ a direct sum of uniform modules?
\end{quote}

Dung \cite{dung1994extending}*{Corollary 9.3} resolved this affirmatively for self-projective modules. The completely extending condition is significantly stronger than purity. Replacing it with the pure extending condition broadens the class of modules while preserving much of the decomposition behavior. In this subsection, we provide a partial answer to the general case under the assumption that all factor modules are pure extending.

Our approach builds on the theory of endomorphism ring conditions. The concepts of endonoetherian and endoartinian modules, originally introduced in a lost preprint by Kaidi and Campos \cite{kaidimodules} and later revisited in \citelist{\cite{gouaid2020endo-noetherian} \cite{gera2026endoartinian}}, play a key role. A module \( M \) is said to be \emph{endonoetherian} (resp., \emph{endoartinian}) if it satisfies the ascending (resp., descending) chain condition on endomorphism kernels (resp., images). These conditions provide a means to control decomposition in terms of endomorphism behavior.

The following two observations underpin our main result:
\begin{enumerate}
    \item A cyclic module whose factor modules are endoartinian is itself endoartinian (see Theorem \ref{cyclic. factor endonoe => endonoe}).
    \item Every proper pure submodule of an indecomposable pure extending module is pure-uniform (see Proposition \ref{indec. pure ext. => uni.}).
\end{enumerate}

These observations together yield a structural generalization of the Osofsky–Smith theorem in the context of pure extending modules.

\subsubsection{Main Theorem and Consequences}

The following theorem identifies purity-based decomposability criteria that persist in the absence of full extending behavior.

\begin{theorem} \label{cyclic fac PE => uni. submod.}
    Let \( M \) be a cyclic module such that proper factor modules of \( M \) is pure extending. Then \( M \) is a finite direct sum of pure-uniform submodules.
\end{theorem}

Before proceeding to the main theorem, we examine two structural properties essential to our proof: one concerning endoartinianity in cyclic modules, and another describing the uniformity of pure submodules in indecomposables. 

\begin{theorem} \label{cyclic. factor endonoe => endonoe}
    Let \( M \) be a cyclic module such that every factor module of \( M \) is endoartinian. Then \( M \) is endoartinian.
\end{theorem}

\begin{proof}
    Suppose $M$ is not endoartinian. Then there exists a strictly descending chain
\[
M = f_0(M) \supsetneq f_1(M) \supsetneq f_2(M) \supsetneq \cdots
\]
of endomorphism images. Let $N = \bigcap_i f_i(M)$. The quotient $M/N$ is cyclic (as a homomorphic image of $M$) and inherits a corresponding chain
\[
M/N \supseteq f_1(M)/N \supseteq f_2(M)/N \supseteq \cdots
\]
of endomorphism images, showing $M/N$ is not endoartinian—a contradiction.
\end{proof}

Recall that a submodule $N$ of $M$ is \emph{pure-essential} if for every pure submodule $P$ of $M$ with $P\cap N=0$, we have $P=0$. A module $M$ is \emph{pure-uniform} if every nonzero submodule of $M$ is pure-essential.

\begin{proposition} \label{indec. pure ext. => uni.}
Let \(M\) be an indecomposable pure extending module.  Then every nonzero pure submodule of \(M\) is pure-essential (equivalently, every pure submodule of \(M\) is pure-uniform).
\end{proposition}

\begin{proof}
Let \(P\le M\) be a nonzero pure submodule and let \(X\leq P\) be a nonzero pure submodule of \(P\). Since \(X\) is pure in \(M\) and \(M\) is pure extending, there is a direct summand \(D\leq^{\oplus}M\) with \(X\leq_e D\). As \(M\) is indecomposable and \(X\neq0\), necessarily \(D=M\). Hence \(X\leq_e M\).

Now let \(Y\le P\) be any nonzero pure submodule of \(P\). Since \(Y\le M\) is nonzero and \(X\leq_e M\), we have \(X\cap Y\neq0\). Thus every nonzero pure submodule \(X\) of \(P\) meets every nonzero pure submodule \(Y\) of \(P\), i.e. \(X\) is pure-essential in \(P\).

As \(X\) was an arbitrary nonzero pure submodule of \(P\), it follows that \(P\) is pure-uniform.
\end{proof}

\begin{remark}
\begin{enumerate}
    \item The indecomposability hypothesis in Proposition \ref{indec. pure ext. => uni.} is essential. For instance, let \( R = \Bbbk[x]/(x^2) \) be a local Frobenius algebra and put \( M = R \oplus R \). Since \( R \) is self-injective, \( M \) is injective (hence pure extending). However, the pure submodule \( P = M \) is not pure-uniform, because the nonzero pure submodules \( R \oplus 0 \) and \( 0 \oplus R \) have zero intersection, so neither is pure-essential in \( M \).

    \item Proposition \ref{indec. pure ext. => uni.} shows that indecomposable pure extending modules exhibit a pure analogue of a well-known property of extending modules: in both cases, submodules behave uniformly with respect to (pure) essentiality. This highlights the structural coherence between the classical and pure settings.
\end{enumerate} \qed
\end{remark}

\begin{theorem} \label{pure ext. endo=> direct sum}
    Let \( M \) be a pure extending endoartinian module. Then \( M \) decomposes as a finite direct sum of pure-uniform submodules.
\end{theorem}

\begin{proof}
    Since \( M \) is endoartinian, $M$ admits a finite decomposition into indecomposable summands i.e., \( M = \bigoplus_i M_i \), where each \( M_i \) is indecomposable (see \cite{gera2026endoartinian}*{Proposition 2.5}). Each \( M_i \) is a direct summand of a pure extending module, hence pure extending by Proposition \ref{summandpurec1}. By Proposition \ref{indec. pure ext. => uni.}, each \( M_i \) is pure-uniform.
\end{proof}

For each $i\in \mathbb{N}$, let $\pur(N_j)$ denote the \emph{purification} (pure closure) of $N_i$ in $N$, i.e., the smallest pure submodule of $N$ containing $N_i$. The purification operator is monotone: if $A\leq B$, then the smallest pure submodule containing $A$ is contained in the smallest pure submodule containing $B$.

\begin{proposition} \label{cyc. fac. PE=> endoart.}
    Let $M$ be a cyclic module such that every cyclic factor module of $M$ is pure extending. Then every cyclic factor module of $M$ is artinian. Moreover, $M$ is endoartinian.
\end{proposition}

\begin{proof}
    Let $N=M/X$ be any cyclic factor module of $M$. By hypothesis, $N$ is cyclic and pure extending. Since $N$ is finitely generated, it has a finite uniform, say $d$. 

    Suppose, toward a contradiction, that $N$ is not artinian. Then there exists an infinite strictly descending chain $N_1 \supsetneq N_2 \supsetneq \cdots$ of submodules of $N$. So, we have $\pur(N_1) \supsetneq \pur(N_2) \supsetneq \cdots$. Each $\pur(N_j)$ is a nonzero pure submodule of $N$. Since $N$ is pure extending, there exists for each $j$ a direct summand $D_j\leq^{\oplus} N$ in which $\pur(N_j)$ is essential. Because $\pur(N_j)$ is essential in $D_j$, they have the same uniform dimension, and thus $\udim(D_j)\leq d$.

    Since $d$ is finite, the sequence $\udim(D_j)$ stabilizes at some integer $k\leq d$. let $J$ be such that for all $j\geq J$, $\udim(D_j)=k$.

    For $j\geq J$, we have $\pur(N_{j+1})\leq \pur (N_j)\leq D_j$. Since $\pur(N_{j+1})$ is essential in $D_{j+1}$ and $\udim(D_{j+1})=\udim(D_j)=k$, it follows that $D_{j+1} \leq D_j$. The equality of uniform dimensions forces $D_{j+1}=D_j$. Denote this common direct summand by $D$.

    For all $j\geq J$, we have $\pur(N_j)\leq D$ and $\pur(N_j)$ is essential in $D$, so $\pur(N_j)=D$. Thus, for $j\geq J$, all $N_j$ are submodules of the cyclic module $D$ whose purification is $D$. Since $D$ is cyclic, the descending chain $N_J\supsetneq N_{J+1} \supsetneq \cdots$ must stabilize, a contradiction.

    Therefore, $N$ is artinian. Since artinian modules are endoartinian, each factor of $M$ is endoartinian. By Theorem \ref{cyclic. factor endonoe => endonoe}, $M$ is endoartinian.
\end{proof}

We now combine these observations to establish the decomposition property for cyclic modules under the pure extending hypothesis.

\begin{proof}[Proof of Theorem \ref{cyclic fac PE => uni. submod.}]
    Let \( M = mR \) be a cyclic module. Since every cyclic factor of \( M \) is pure extending by hypothesis, Proposition \ref{cyc. fac. PE=> endoart.} implies that \( M \) is endoartinian. Applying Theorem \ref{pure ext. endo=> direct sum}, we conclude that \( M \) is a finite direct sum of pure-uniform submodules.
\end{proof}

\begin{corollary} \label{coropen1}
    Let \( R \) be a von Neumann regular ring, and let \( M \) be a cyclic right \( R \)-module such that every cyclic factor module of \( M \) is extending. Then \( M \) is a finite direct sum of uniform submodules.
\end{corollary}

\begin{proof}
    Over a von Neumann regular ring, every submodule is pure, so extending and pure extending coincide. The result follows from Theorem \ref{cyclic fac PE => uni. submod.}.
\end{proof}

\begin{note}
    The von Neumann regularity assumption in Corollary \ref{coropen1} appears to be essential, though constructing explicit counterexamples demonstrating its necessity has proven difficult. \qed
\end{note}

\subsection{Connections to Morphic Modules} \label{3.2}

\subsubsection{Rickart Modules and Morphic Structures: Definitions and Context}

The study of morphic and quasi-morphic modules connects naturally to structural decompositions involving purity and Rickart-type properties. Rickart and dual Rickart modules were systematically introduced and studied by Lee, Rizvi, and Roman in \citelist{\cite{lee2010rickart} \cite{lee2011dual} \cite{lee2012direct}}. A module \( M \) is called \emph{Rickart} (resp., \emph{dual Rickart} or d-Rickart) if for every \( f \in S = \morp(M) \), we have \( \ker(f) = eM \) (resp., \( \im(f) = eM \)) for some idempotent \( e^2 = e \in S \). It is known that \( M \) is \emph{endoregular} if and only if it is both Rickart and d-Rickart \cite{lee2013modules}*{Theorem 1.1}.

However, the Rickart and d-Rickart properties are not, in general, preserved under direct sums. To address this, Lee and Bárcenas introduced the notion of \emph{\(\Sigma\)-Rickart} modules in \cite{lee2020sigma}, requiring the Rickart condition to hold for arbitrary direct sums of copies of the module. The dual notion, \emph{\(\Sigma\)-d-Rickart} modules, was studied by Kumar and Gupta in \cite{kumar2024sigma}. Specifically, a module \( M \) is said to be \(\Sigma\)-Rickart (resp., \(\Sigma\)-d-Rickart) if every direct sum of copies of \( M \) is Rickart (resp., d-Rickart); equivalently, for any set \( I \) and homomorphism \( f \in \morp(M^I) \), there exists a finite subset \( J \subseteq I \) such that \( \ker(f) \leq^{\oplus} M^J \) (resp., \( \im(f) \leq^{\oplus} M^J \)). 

These notions relate closely to decomposition-theoretic conditions, particularly the \( C_2 \) and \( D_2 \) conditions. Recall that a module \( M \) satisfies the \( C_2 \) condition if every submodule isomorphic to a direct summand is itself a direct summand, and the \( D_2 \) condition if the intersection of any two direct summands is again a direct summand. We define:

\begin{align*}
\Sigma\text{-}C_2 &: \text{Every direct sum of copies of \( M \) satisfies } C_2; \\
\Sigma\text{-}D_2 &: \text{Every direct sum of copies of \( M \) satisfies } D_2.
\end{align*}

It is immediate that every \(\Sigma\)-Rickart module satisfies \(\Sigma\)-\(D_2\), and every \(\Sigma\)-d-Rickart module satisfies \(\Sigma\)-\(C_2\) (see Lemma \ref{SR(S-d-R)=>S-D2(S-C2)}).

These structural ideas naturally relate to stronger conditions on endomorphism rings. A module $M$ is called \emph{strongly $\pi$-endoregular} if, for every $f\in S=\morp_R(M)$, there exists $n\geq 1$ such that $\im(f^n)=\im(f^{n+1})$ and $\ker (f^n)=\ker (f^{n+1})$. This notion is inspired by the theory of strongly $\pi$-regular modules introduced by Azumaya \cite{azumaya1954strongly} and further developed by Armendariz \cite{armendariz1978injective}. We prove that a module is strongly $\pi$-endoregular if and only if it is both abelian and strongly $\pi$-regular (see Proposition~\ref{strongly pi-endoregular <=> abelian+strongly endoregular}), paralleling \cite{dehghani2025centrally}*{Theorem 3.1}. Moreover, in Lemma~\ref{strongly pi-endo <=> S-Rickart+S-d-Rickart}, we show that the class of strongly $\pi$-endoregular modules coincides with those that are simultaneously $\Sigma$-Rickart and $\Sigma$-d-Rickart modules, provided the module has finite uniform dimension.

In Theorem~\ref{nonsingular PE <=> S-R-noeth}, we generalize Lee and Barcenas' result \cite{lee2020sigma}*{Example 2.2(vi)} by showing that if $R$ is right noetherian and $M$ is a finitely generated, nonsingular, pure extending right $R$-module, then $M$ is $\Sigma$-Rickart. Our theorem replaces injectivity with the substantially weaker assumption of finite generation and pure extensibility, thereby extending the $\Sigma$-Rickart property beyond injective contexts. However, the converse implication remains open—even over right noetherian rings—for finitely generated $\Sigma$-Rickart modules whose nonsingular pure extending structure is not yet characterized.

Turning to morphic and quasi-morphic modules, Camillo and Nicholson introduced \emph{left quasi-morphic} rings in \cite{camillo2007quasi}, defined by the condition that for each \( a \in R \), there exists \( b, c \in R \) such that \( \lann_R(a) = Rb \) and \( Ra = \lann_R(c) \). If one can choose \( b = c \), the ring is \emph{left morphic}, a concept earlier studied by Nicholson and Campos \cite{nicholson2004rings}. These notions were extended to modules by Nicholson and Campos \cite{nicholson2005morphic}, and further to the quasi-morphic case by An, Nam, and Tung \cite{an2016quasi}. A module \( M \) is quasi-morphic if, for every \( f \in \morp (M) \), there exist \( g, h \in \morp (M) \) such that \( \ker(f) = \im(g) \) and \( \im(f) = \ker(h) \); if \( g = h \), then \( M \) is morphic.

Recently, Dehghani and Sedaghatjoo introduced \emph{centrally morphic} and \emph{centrally quasi-morphic} modules in \cite{dehghani2025centrally}, in which the elements \( g, h \) above are required to lie in the center \( \cent(\morp (M)) \). These definitions yield a hierarchy of morphic-type module classes, connected via the following implications:

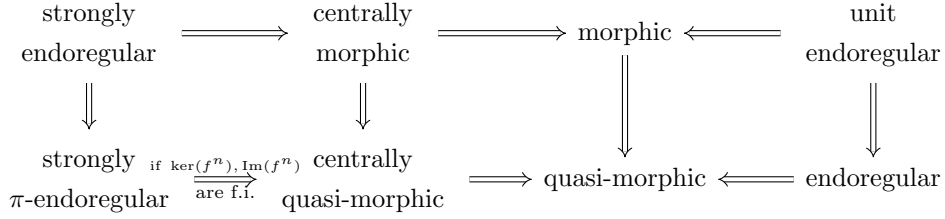
\begin{center}
\begin{tikzcd}
    \begin{array}{c}
         \text{strongly} \\
         \text{endoregular}
    \end{array} \ar[r,Rightarrow] \ar[d,Rightarrow] & \begin{array}{c}
          \text{centrally} \\
         \text{morphic}
    \end{array} \ar[r,Rightarrow] \ar[d,Rightarrow] & \text{morphic} \ar[d,Rightarrow] & \begin{array}{c}
         \text{unit} \\
         \text{endoregular}
    \end{array} \ar[l,Rightarrow] \arrow[d,Rightarrow] \\
    \begin{array}{c}
         \text{strongly} \\
         \pi\text{-endoregular}
    \end{array}
    \arrow[r, Rightarrow,
  "\tiny \substack{\text{if }\ker(f^n),\,\im(f^n)}" above,
  "\substack{\text{are f.i.}}" below] &
    \begin{array}{c}
         \text{centrally} \\
         \text{quasi-morphic}
    \end{array} \ar[r,Rightarrow] & \text{quasi-morphic} & \text{endoregular} \ar[l,Rightarrow]
\end{tikzcd}
\end{center}

Here ``f.i." stands for fully invariant. None of the implications is reversible in general. However, within certain structural settings, stronger equivalences can be recovered. The motivation for this investigation is to address the following open question posed in \cite{dehghani2025centrally}*{Question 2.14}:

\begin{quote}
Is every centrally quasi-morphic module centrally morphic?
\end{quote}

We answer this question affirmatively under sufficient conditions, notably for finitely generated, nonsingular, pure-extending modules over noetherian rings (see Proposition~\ref{propopen} and Corollary~\ref{coropen}), and for modules over semisimple artinian rings (see Corollary~\ref{coropen}). We also show that centrally morphic modules need not be Rickart (see Example~\ref{ex:cm_not_rickart}). The general question of whether every centrally quasi-morphic module is centrally morphic remains open, and our results do not settle it in full generality. With these clarifications established, we next investigate how these refined notions integrate into the broader morphic hierarchy and its endoregular refinements. With these clarifications established, we next investigate how these refined notions integrate into the broader morphic hierarchy and its endoregular refinements.

\subsubsection{On Centrally Quasi-Morphic versus Centrally Morphic Modules}

\begin{proposition} \label{CQM}
Let $M$ be a right $R$-module and $S = \morp_R(M)$. Suppose $M$ is strongly $\pi$-endoregular and, for every $f \in S$, there exists $n \ge 1$ such that the summands $\ker(f^n)$ and $\im(f^n)$ are fully invariant submodules of $M$. Then $M$ is centrally quasi-morphic.
\end{proposition}

\begin{proof}
Fix $f \in S$ and choose $n \ge 1$ such that
\[
M = \ker(f^n) \oplus \im(f^n),
\]
with both summands fully invariant.  

Define the linear maps $g: M \to \ker(f^n), g(k+i) := k$ and $h: M \to \im(f^n), \quad h(k+i) := i$ for each $x = k + i \in \ker(f^n) \oplus \im(f^n)$.  Since the summands are fully invariant, $g$ and $h$ commute with every $s \in S$, so $g, h \in \cent(S)$.  

By construction, $\im(g) = \ker(f^n)$ and $\ker(h) = \im(f^n)$, so the sequence
\[
M \xrightarrow{g} M \xrightarrow{f^n} M \xrightarrow{h} M
\]
is exact. Since $f$ was arbitrary, $M$ satisfies the centrally quasi-morphic condition.
\end{proof}

\begin{example}
Let $\Bbbk$ be a field with $\mathrm{char}(\Bbbk) \neq 2$, and let $M = \Bbbk^2 = \left\{ \begin{pmatrix} x \\ y \end{pmatrix} \mid x,y \in \Bbbk \right\}$ be the right $\Bbbk$-module of column vectors, with endomorphism ring $S = \morp_\Bbbk(M) \cong M_2(\Bbbk)$.

Consider the idempotent endomorphism $f = \begin{pmatrix} 1 & 1 \\ 0 & 0 \end{pmatrix} \in S,$ so that $f^2 = f$. Then $M$ is strongly $\pi$-endoregular with $n=1$, giving the decomposition $M = \ker(f) \oplus \im(f)$, where $v_1 = \begin{pmatrix} 1 \\ -1 \end{pmatrix}$, $\ker(f) = \mathrm{span}\{v_1\}$, and $v_2 = \begin{pmatrix} 1 \\ 0 \end{pmatrix}$, $\im(f) = \mathrm{span}\{v_2\}$.

Every vector $v = (x,y)^\top \in M$ decomposes uniquely as $v = \alpha v_1 + \beta v_2$, with $\alpha = -y, \ \beta = x+y$.

The canonical projections onto these summands are
\[
g: M \to \ker(f), \quad g(v) = \alpha v_1 = (-y, y)^\top,
\]
\[
h: M \to \im(f), \quad h(v) = \beta v_2 = (x+y, 0)^\top.
\]

Consider the endomorphism $t = \begin{pmatrix} 0 & 1 \\ 0 & 0 \end{pmatrix} \in S$. Then the commutator is
\[
[g,t] = g t - t g = \begin{pmatrix} 0 & -1 \\ 1 & 1 \end{pmatrix} \neq 0,
\]
so $g \notin \cent(S)$. Similarly, $h \notin \cent (S)$.

To see that the summands are not fully invariant, consider $r = \begin{pmatrix} 1 & 0 \\ 1 & 0 \end{pmatrix} \in S$. Then $r(v_1) = \begin{pmatrix} 1 \\ 1 \end{pmatrix} \notin \ker(f)$, $r(v_2) = \begin{pmatrix} 1 \\ 1 \end{pmatrix} \notin \im(f)$, so neither $\ker(f)$ nor $\im(f)$ is fully invariant.

Although the sequence
\[
M \xrightarrow{g} M \xrightarrow{f} M \xrightarrow{h} M
\]
is exact, the projections $g$ and $h$ are not central. This example demonstrates that strong $\pi$-endoregularity alone does not imply the centrally quasi-morphic property, and the hypothesis requiring the summands $\ker(f^n)$ and $\im(f^n)$ to be fully invariant is essential. \qed
\end{example}

Recall that a ring $R$ is called abelian if each of its idempotents is central. Moreover, an $R$-module $M$ is said to be abelian if $S=\morp_R(M)$ is an abelian ring.

\begin{proposition} \label{strongly pi-endoregular <=> abelian+strongly endoregular}
Let $M$ be an $R$-module and $S=\morp_R(M)$. The following statements are equivalent:
\begin{enumerate}
  \item $M$ is strongly $\pi$-endoregular;
  \item $M$ is abelian and strongly $\pi$-regular.
  \item For every $f\in S$ there exists $n\ge 1$ such that $\im(f^n)=\im(f^{n+1})$ and $\ker(f^n)=\ker(f^{n+1})$
  \item For every $f\in S$ there exists $n\ge 1$ such that $M = \ker(f^n)\oplus \im(f^n)$.
\end{enumerate}
\end{proposition}

\begin{proof}
(1) $\Rightarrow$ (2): If $M$ is strongly $\pi$-endoregular, then $S$ is strongly $\pi$-regular by definition, and reduced. In a reduced strongly $\pi$-regular ring, all idempotents are central, i.e., $S$ is abelian.

(2) $\Rightarrow$ (3): Let $f\in S$. By strong $\pi$-regularity, there exists $n$ and $g\in S$ such that $f^n = f^{n+1} g$. Then $f^n$ satisfies $f^n S = f^{n+1} S$ and $S f^n = S f^{n+1}$. In particular, $\im(f^n)=\im(f^{n+1})$ and $\ker(f^n)=\ker(f^{n+1})$.

(3) $\Rightarrow$ (4): Let $f\in S$ and $n$ be as in (3). Consider $f^n: M \to \im(f^n)$. Then $f^n|_{\im(f^n)}$ is surjective onto $\im(f^n)$. The stabilization $\ker(f^n) = \ker(f^{2n})$ ensures $\im(f^n)\cap \ker(f^n)=0$ (otherwise some nonzero element would be annihilated by a power of $f^n$). Hence $M = \ker(f^n)\oplus \im(f^n)$.

(4) $\Rightarrow$ (1): Let $f\in S$ and $n$ be such that $M = \ker(f^n)\oplus \im(f^n)$. Then $f^n|_{\im(f^n)}$ is an automorphism of $\im(f^n)$, so $f^n$ is von Neumann regular in $S$. If $s\in S$ were nilpotent, then for sufficiently large $n$ we would have $\im(s^n)=0$ and $\ker(s^n)=M$, contradicting the direct sum decomposition unless $s=0$. Thus $S$ is reduced, completing the proof.
\end{proof}

\begin{corollary} \label{CQM1}
    Let $M$ be a centrally quasi-morphic module. If $M$ is strongly $\pi$-regular, then it is strongly $\pi$-endoregular.
\end{corollary}

\begin{proof}
    Since $M$ is centrally quasi-morphic, it is abelian by \cite{dehghani2025centrally}*{Proposition 2.4}. From Proposition \ref{strongly pi-endoregular <=> abelian+strongly endoregular}, an abelian strongly $\pi$-regular module is strongly $\pi$-endoregular.
\end{proof}

We now examine how purity and endomorphism conditions interact to determine morphic behavior, particularly within centrally quasi-morphic modules.

\begin{lemma}\label{SR(S-d-R)=>S-D2(S-C2)}
Let $M$ be a right $R$-module. The following statements hold:
\begin{enumerate}
  \item If $M$ is $\Sigma$-Rickart and each direct sum $M^I$ satisfies the $C_2$ condition, then every $M^I$ satisfies $D_2$; hence $M$ satisfies $\Sigma\text{-}D_2$.
  \item If $M$ is $\Sigma$-d-Rickart, then every $M^I$ satisfies $C_2$; hence $M$ satisfies $\Sigma\text{-}C_2$.
\end{enumerate}
\end{lemma}

\begin{proof}
(1) Let $I$ be any index set and $f\in\morp(M^I)$. Since $M$ is $\Sigma$-Rickart, $\ker f$ is a direct summand of $M^I$; write
\[
M^I=\ker f\oplus L.
\]
Then $\im f\cong M^I/\ker f\cong L$, so $\im f$ is isomorphic to a direct summand of $M^I$. By the hypothesis that $M^I$ satisfies the $C_2$ condition, we conclude $\im f\leq^{\oplus}M^I$.

Thus every image of an endomorphism of $M^I$ is a direct summand. We now show $M^I$ satisfies $D_2$. Let $A$ and $B$ be direct summands of $M^I$, and let $f,g\in\morp(M^I)$ be the idempotent projections with $\im f=A$ and $\im g=B$. Consider the endomorphism
\[
h:=(1-f)g\in\morp(M^I).
\]
By the previous paragraph $\ker h$ is a direct summand of $M^I$. Observe that $g(\ker h)=A\cap B$: indeed, if $k\in\ker h$ then $(1-f)g(k)=0$ so $g(k)\in A$, hence $g(k)\in A\cap B$; conversely if $a\in A\cap B$ then $a=g(a)$ and $(1-f)g(a)=(1-f)(a)=0$, so $a\in g(\ker h)$. Since $\ker h$ is a direct summand of $M^I$, its image under the idempotent $g$ is a direct summand of $\ker h$, hence a direct summand of $M^I$. Therefore $A\cap B=g(\ker h)$ is a direct summand of $M^I$. As $A$ and $B$ were arbitrary summands, $M^I$ satisfies $D_2$. Because $I$ was arbitrary, $M$ satisfies $\Sigma\text{-}D_2$.

(2) Let $I$ be an index set and suppose $N\le M^I$ is a submodule isomorphic to a direct summand $S\le M^I$. Let $\varphi:N\stackrel{\cong}{\longrightarrow}S$ be an isomorphism, let $\iota_N:N\hookrightarrow M^I$ be the inclusion, and let $\pi_S:M^I\twoheadrightarrow S$ be the projection onto $S$ (exists since $S$ is a summand). Define
\[
h:=\iota_N\circ\varphi^{-1}\circ\pi_S\in\morp  (M^I).
\]
Then $\im h = \iota_N(\varphi^{-1}(\pi_S(M^I)))=\iota_N(N)=N$. By the $\Sigma$-d-Rickart property of $M$ there is a finite $J\subseteq I$ with $N=\im h\leq^{\oplus}M^J$. Since $M^J$ is a direct summand of $M^I$ (indeed $M^I=M^J\oplus M^{I\setminus J}$), it follows that $N$ is a direct summand of $M^I$. Hence every submodule of $M^I$ which is isomorphic to a summand is itself a summand; that is, $M^I$ satisfies $C_2$. Because $I$ was arbitrary, $M$ satisfies $\Sigma\text{-}C_2$.
\end{proof}

\begin{lemma}\label{strongly pi-endo <=> S-Rickart+S-d-Rickart}
Let $M$ be a module having finite uniform dimension. Then $M$ is strongly $\pi$-endoregular if and only if it is both $\Sigma$-Rickart and $\Sigma$-d-Rickart.
\end{lemma}

\begin{proof}
($\Rightarrow$). Assume $M$ is strongly $\pi$-endoregular, so $M=\ker(f^n)\oplus\im(f^n)$ by Proposition \ref{strongly pi-endoregular <=> abelian+strongly endoregular} for $f\in \morp_R (M)$. Let $I$ be any index set and $g\in\morp_R(M^I)$. The endomorphism ring $\morp_R(M^I)$ is isomorphic to the ring of column-finite matrices over $S$. Since $S$ is strongly $\pi$-regular, this matrix ring is also strongly $\pi$-regular. Thus there exist $m\ge1$ and $h\in\morp_R(M^I)$ such that $g^m = g^m h g^m$. Let $\varepsilon = g^m h$. Then $\varepsilon^2 = \varepsilon$ and $M^I = \ker(g^m) \oplus \im(g^m)$.

Since $\ker(g) \subseteq \ker(g^m)$ and both are direct summands (as $M^I$ is Rickart), $\ker(g)$ is a direct summand. Similarly, $\im(g)$ is a direct summand as it contains $\im(g^m)$ and both are direct summands. Hence $M$ is $\Sigma$-Rickart and $\Sigma$-d-Rickart.

($\Leftarrow$). Assume $M$ is both $\Sigma$-Rickart and $\Sigma$-d-Rickart, and suppose moreover that $M$ has finite uniform dimension. By Lemma \ref{SR(S-d-R)=>S-D2(S-C2)}, every $M^I$ satisfies $C_2$ and $D_2$. Fix $f\in S$. The ascending chain of kernels
\[
\ker(f)\subseteq\ker(f^2)\subseteq\cdots
\]
consists of direct summands (by $\Sigma$-Rickart), and the descending chain of images
\[
\im(f)\supseteq\im(f^2)\supseteq\cdots
\]
consists of direct summands (by $\Sigma$-d-Rickart). Finite uniform dimension implies there can be no infinite strictly descending chain of submodules, hence both chains stabilize: there exist $m,n$ with $\ker(f^m)=\ker(f^{m+1})$ and $\im(f^n)=\im(f^{n+1})$. With $N=\max\{m,n\}$ the usual Fitting argument yields
\[
M=\ker(f^N)\oplus\im(f^N),
\]
so $M$ is strongly $\pi$-endoregular by Proposition \ref{strongly pi-endoregular <=> abelian+strongly endoregular}.
\end{proof}

\begin{proposition} \label{S-R <=> S-d-R}
Let $M$ be a centrally quasi-morphic module. Then $M$ is $\Sigma$-Rickart if and only if it is $\Sigma$-d-Rickart.
\end{proposition}

\begin{proof}
Since $M$ is centrally quasi-morphic, $M$ is abelian by \cite{dehghani2025centrally}*{Proposition 2.4}. 

($\Rightarrow$). Assume $M$ is $\Sigma$-Rickart. Since $M$ is abelian, each $M^I$ satisfies $C_2$. By Lemma \ref{SR(S-d-R)=>S-D2(S-C2)}(1), $M$ satisfies $\Sigma$-$D_2$. For any $f \in \morp_R(M^I)$, $\ker f \leq^\oplus M^I$, so $\im f \cong M^I/\ker f$ is isomorphic to a direct summand. By $D_2$, $\im f \leq^\oplus M^I$. Thus $M$ is $\Sigma$-d-Rickart.

($\Leftarrow$). Assume $M$ is $\Sigma$-d-Rickart. For any $f \in \morp_R(M^I)$, $\im f \leq^\oplus M^I$. Let $e \in \morp_R(M^I)$ be the idempotent with $\im e = \im f$. Since $M$ is abelian, $e$ is central. We prove $\ker f = \im(1-e)$:
\begin{itemize}
    \item $\im(1-e) \subseteq \ker f$: For $x = (1-e)(y)$, $f(x) = f(1-e)(y) = (f - fe)(y) = 0$ since $f = fe$.
    \item $\ker f \subseteq \im(1-e)$: If $f(x) = 0$, then $f(e(x)) = 0$, so $e(x) \in \ker f \cap \im f = 0$, hence $x = (1-e)(x) \in \im(1-e)$.
\end{itemize}
Thus $\ker f = \im(1-e) \leq^\oplus M^I$, so $M$ is $\Sigma$-Rickart.
\end{proof}

Proposition \ref{S-R <=> S-d-R} shows that within the class of centrally quasi-morphic modules, the asymmetry between kernel and image behavior observed in general \(\Sigma\)-Rickart versus \(\Sigma\)-d-Rickart modules collapses. This provides a natural setting where these two classes coincide.

\begin{theorem} \label{main1}
Let \(M\) be a \(\Sigma\)-Rickart module of finite uniform dimension. Then the following statements are equivalent:
\begin{enumerate}
    \item \(M\) is \(\Sigma\)-d-Rickart;
    \item \(M\) satisfies the \(\Sigma\)-$C_2$ condition;
    \item \(M\) is strongly \(\pi\)-endoregular.
\end{enumerate}
Moreover, if the kernel and image summands of powers of endomorphisms of \(M\) are fully invariant, then \(M\) is centrally quasi-morphic.
\end{theorem}

\begin{proof}
    (1) $\Leftrightarrow$ (2). For any \(f^n \in \morp_R(M)\), $\Sigma$-Rickart ensures \(\ker(f^n)\) is a direct summand. If \(M\) satisfies $\Sigma$-$C_2$, then \(\im(f^n) \cong M/\ker(f^n)\) is also a direct summand, so \(M\) is $\Sigma$-d-Rickart. Conversely, if \(M\) is $\Sigma$-d-Rickart, then Lemma \ref{SR(S-d-R)=>S-D2(S-C2)}(2) guarantees $\Sigma$-$C_2$. Hence, $M$ is $\Sigma$-d-Rickart if and only if it satisfies $\Sigma$-$C_2$.

    (1) $\Leftrightarrow$ (3). Since \(M\) has finite uniform dimension and is both $\Sigma$-Rickart and $\Sigma$-d-Rickart, Lemma \ref{strongly pi-endo <=> S-Rickart+S-d-Rickart} implies \(M\) is strongly $\pi$-endoregular.

    If the kernel and image summands of powers of endomorphisms are fully invariant, then Proposition \ref{CQM} applies, and \(M\) is centrally quasi-morphic.
\end{proof}

\begin{theorem} \label{main2}
Let \(M\) be a \(\Sigma\)-d-Rickart module of finite uniform dimension. Then the following statements are equivalent:
\begin{enumerate}
    \item \(M\) is \(\Sigma\)-Rickart;
    \item \(M\) satisfies the \(\Sigma\)-$D_2$ condition;
    \item \(M\) is strongly \(\pi\)endoregular.
\end{enumerate}
Moreover, if the kernel and image summands of powers of endomorphisms of \(M\) are fully invariant, then \(M\) is centrally quasi-morphic.
\end{theorem}

The proof is dual to the $\Sigma$-Rickart version.

\begin{proof}
(1) $\Leftrightarrow$ (2). $\Sigma$-d-Rickart ensures \(\im(f^n)\) is a direct summand for all \(f^n \in \morp_R(M)\). If \(M\) satisfies $\Sigma$-$D_2$, then \(\ker(f^n) \cong M/\im(f^n)\) is also a direct summand, so \(M\) is $\Sigma$-Rickart. Conversely, a $\Sigma$-Rickart module satisfying $\Sigma$-$C_2$ implies $\Sigma$-$D_2$ by Lemma \ref{SR(S-d-R)=>S-D2(S-C2)}(1). Hence, $M$ is $\Sigma$-Rickart if and only if it satisfies $\Sigma$-$D_2$.

The rest of the proof is similar to Theorem \ref{main1}.
\end{proof}

\begin{theorem} \label{nonsingular PE <=> S-R-noeth}
Let $R$ be a right noetherian ring and let $M$ be a finitely generated, nonsingular, pure extending right $R$-module. Then $M$ is $\Sigma$-Rickart.
\end{theorem}

\begin{proof}
Let $R$ be a right noetherian ring and $M$ a finitely generated, nonsingular, pure extending right $R$-module. To prove that $M$ is $\Sigma$-Rickart, it suffices to show that $M^{(I)}$ is Rickart for every index set $I$. Fix such an $I$ and set $X = M^{(I)}$. Since $R$ is right noetherian and $M$ is finitely generated, the module $M$ is noetherian, whence $X$ is locally noetherian and locally finitely presented. For an arbitrary $f \in \morp_R(X)$, let $K = \ker f$ and denote by $\overline{K}=\pur(K)$ the purification of $K$ in $X$. Because $M$ is pure extending and $R$ is right noetherian, the module $X$ also inherits the pure extending property. Consequently, $\overline{K}$ is essential in a direct summand $D$ of $X$, so that $K \leq \overline{K} \leq D$ and $X = D \oplus D'$ for some submodule $D'$.

Suppose first that $\overline{K} \ne K$, and choose $x \in \overline{K} \setminus K$. Since $X$ is locally finitely presented, there exists a finitely generated submodule $F \le X$ with $x \in F$. As $\overline{K}$ is pure in $X$, the intersection $F \cap \overline{K}$ is pure in $F$, and because $F$ is finitely presented, purity implies that $F \cap \overline{K}$ is a direct summand of $F$, say $F = (F \cap \overline{K}) \oplus C$ for some $C \le F$. The image of $x$ in $(F \cap \overline{K})/(F \cap K)$ is then nonzero, so $(F \cap \overline{K})/(F \cap K)$ is a nonzero finitely presented submodule of $F/(F \cap K)$. The inclusion $F \hookrightarrow X$ induces an embedding $F/(F \cap K) \hookrightarrow X/K$, and hence $X/K$ contains a nonzero finitely presented submodule. Since $M$ is nonsingular and $R$ is right noetherian, the module $X = M^{(I)}$ is nonsingular, and so is its quotient $X/K$. However, a nonsingular module cannot contain a nonzero finitely presented singular submodule. This contradiction forces $\overline{K} = K$, and thus $K$ is pure in $X$ and essential in $D$.

To complete the argument, let $y \in D$ and choose a finitely generated submodule $F \le D$ with $y \in F$. Because $M$ is noetherian, every finitely generated submodule of $D$ is finitely presented. Then $F \cap K$ is pure in $F$ (since $K$ is pure in $X$) and essential in $F$ (since $K$ is essential in $D$). Purity now implies that $F \cap K$ is a direct summand of $F$, and being essential, it must coincide with $F$. Hence $F \leq K$, and therefore $y \in K$. As $y$ was arbitrary, it follows that $D \leq K$, so $K = D$. Consequently, $\ker f = K$ is a direct summand of $X$. Since $f \in \morp_R(X)$ was arbitrary, $X$ is Rickart. As this holds for all index sets $I$, we conclude that $M$ is $\Sigma$-Rickart.
\end{proof}

\begin{proposition}\label{propopen}
Let $R$ be a noetherian ring, and let \( M \) be a finitely generated, nonsingular, pure extending module. Then \( M \) is centrally quasi-morphic if and only if it is centrally morphic.
\end{proposition}

\begin{proof}
    ($\Leftarrow$). Straightforward that every centrally morphic module is centrally quasi-morphic.
    
    ($\Rightarrow$). Assume $M$ is finitely generated, nonsingular, and pure extending. By Theorem \ref{nonsingular PE <=> S-R-noeth}, $M$ is $\Sigma$-Rickart. Let $S=\morp_R(M)$. From \cite{lee2020sigma}*{Proposition 4.3}, $S$ is right semihereditary. Since $S$ is right semihereditary, for each $f\in S$, the right ideal $fS$ is projective. Consider the inclusion map $i:fS\hookrightarrow S$. Since $fS$ is projective, this inclusion splits: there exists an $S$-linear map $p:S\to fS$ such that $p\circ i=\id_{fS}$. Let $e=i(p(1))$. It is easy to see that $e$ is idempotent and $fS=eS$. Therefore, every principal right ideal of $S$ is generated by an idempotent, which is equivalent to $S$ being von Neumann regular; consequently, $M$ is endoregular. Finally, since $M$ is endoregular and centrally quasi-morphic, \cite{dehghani2025centrally}*{Proposition 3.4} shows that $M$ is strongly endoregular, and \cite{dehghani2025centrally}*{Corollary 3.2}, then gives that every strongly endoregular module is centrally morphic.
\end{proof}

This leads to the following corollary, which resolves \cite{dehghani2025centrally}*{Question 2.14} in the setting of extending modules.

\begin{corollary} \label{coropen}
    Let \( R \) be a semisimple artinian ring, and let \( M \) be a finitely generated \( R \)-module. Then \( M \) is centrally morphic if and only if it is centrally quasi-morphic.
\end{corollary}

\begin{proof}
Since $R$ is semisimple artinian, it is von Neumann regular and (left and right) noetherian; moreover every finitely generated $R$-module is semisimple, hence nonsingular, extending and pure–extending. In particular the hypotheses of Proposition \ref{propopen} are satisfied for $M$. Applying Proposition \ref{propopen} yields the equivalence: $M$ is centrally quasi-morphic if and only if $M$ is centrally morphic, as required.
\end{proof}

We now show that a centrally morphic module need not be Rickart.

\begin{example} \label{ex:cm_not_rickart}
Let $\Bbbk$ be a field and set $R = \Bbbk[x]/(x^2)$. Consider $M = R$ as a right module over itself. First, $M$ is centrally morphic. Indeed, $\morp_R(M) \cong R$, which is commutative, so every endomorphism is central. For any $f \in \operatorname{End}_R(M)$, there exists $a \in R$ such that $f$ is multiplication by $a$. 
\begin{itemize}
    \item If $a$ is a unit, then $\ker(f)=0$ and $\im(f)=M$. Choose $g=0$.
    \item If $a=0$, then $\ker(f)=M$ and $\im(f)=0$. Choose $g=1$.
    \item If $a = \lambda \bar{x}$ with $\lambda \neq 0$, then $\ker(f)=\bar{x}R$ and $\im(f)=\bar{x}R$. Choose $g=\bar{x}$.
\end{itemize}
In each case, $g$ is central in $\morp_R(M)$ and satisfies $\ker(f)=\im(g)$ and $\im(f)=\ker(g)$. Hence $M$ is centrally morphic.

However, $M$ is not Rickart. Take $f$ to be multiplication by $\bar{x}$. Then $\ker(f) = \bar{x}R$. Since $R$ is a local ring and indecomposable as a right module over itself, its only direct summands are $0$ and $R$. Because $\bar{x}R$ is a nonzero proper submodule, it is not a direct summand of $R_R$. Thus $\ker(f)$ is not a direct summand of $M$, so $M$ fails the Rickart condition.
\end{example}

\section*{Appendix A: An Example of RD-pure Extending Module which is not Pure Extending}
\renewcommand{\thesection}{A}
\setcounter{theorem}{0}

\begin{example} \label{RD-pure not pure}
Let $R=\mathbb{Z}$ and fix a prime $p$. Put
\[
M \;=\; \mathbb{Z}(p^\infty)\oplus\mathbb{Z}/p\mathbb{Z}=:T\oplus F,
\]
where $T=\mathbb{Z}(p^\infty)$ is the Pr\"{u}fer $p$-group and $F\cong\mathbb{Z}/p\mathbb{Z}$. We show that $M$ is RD-pure extending but not pure extending.

\emph{Claim 1. $M$ is not pure extending}. Choose $y\in T$ of order $p$ and let $x$ generate $F$. Set
\[
U=\langle (y,x)\rangle\le M,
\]
so $U\cong\mathbb{Z}/p\mathbb{Z}$. We first check that $U$ is pure in $M$.

Fix $n\in\mathbb{Z}$. If $p\nmid n$ then multiplication by $n$ is an automorphism on each $p$-primary summand, so $nU=U$ and $nM=M$, hence $nM\cap U=nU$. If $p\mid n$ then $nU=0$, while $nM=pM=T\oplus 0$, because $n$ annihilates $F$. But $(T\oplus 0)\cap U=0$ (any nonzero element of $U$ has nonzero second coordinate), so $nM\cap U=0=nU$. Thus $nM\cap U=nU$ for every $n$, and $U$ is pure in $M$.

Next, observe that the proper direct summands of $M$ are $T\oplus0$ and $0\oplus F$. The projection of $U$ to each summand is nonzero, so $U$ is not contained in either proper summand. Hence, the only direct summand of $M$ in which $U$ could be essential is $M$ itself. But $(0\oplus F)\cap U=0$, so $U$ is not essential in $M$. Therefore, $U$ is a pure submodule of $M$ that is not essential in any direct summand, and $M$ is not pure extending.

\emph{Claim 2. $M$ is RD-pure extending}. Let $N\le M$ be an RD-pure submodule. We will show that $N$ is essential in a direct summand of $M$.

\begin{lemma}
Let $R$ be a ring and $\pi\colon M\to A$ a split projection with section $s\colon A\to M$. If $N\le M$ is RD-pure, then $\pi(N)\le A$ is RD-pure.
\end{lemma}

\begin{proof}
Write $M = A \oplus B$, so that $\pi(a,b)=a$. Then $rM = rA \oplus rB$ and $\pi(rM)=rA$ for all $r\in R$. Since $N$ is RD-pure in $M$, we have $rN = rM \cap N$. 

For $a\in rA\cap \pi(N)$, choose $b\in B$ with $(a,b)\in N$ and $a=ra'$ for some $a'\in A$. Then $(a,b)=(ra',b)\in rM\cap N=rN$, so $(a,b)=r(c,d)$ for some $(c,d)\in N$. Thus $a=rc\in r\pi(N)$, proving $rA\cap\pi(N)\subseteq r\pi(N)$. The reverse inclusion is immediate, hence $r\pi(N)=rA\cap\pi(N)$. Therefore $\pi(N)$ is RD-pure in $A$.
\end{proof}

Applying the lemma to the coordinate projections $\pi_T,\pi_F$, and using the structure of the summands, we obtain:
\[
\pi_T(N)\in\{0,T\},\qquad \pi_F(N)\in\{0,F\},
\]
since the only RD-pure subgroups of $F$ are $0$ and $F$, and any nonzero RD-pure subgroup of the divisible group $T$ must equal $T$ (divisible groups are RD-injective, hence split).

Thus, every nonzero RD-pure $N$ falls into one of the three types below; we treat each.

\noindent\textit{Case 1.} $\pi_T(N)=T,\ \pi_F(N)=0$. \\
Then $N\leq T\oplus0$. Since $T$ is injective (divisible), any nonzero RD-pure submodule of $T$ splits and therefore equals $T$. Hence, $N$ is essential in the direct summand $T\oplus0$.

\noindent\textit{Case 2.} $\pi_T(N)=0,\ \pi_F(N)=F$. \\
Then $N\leq 0\oplus F$; because $F$ is simple, either $N=0$ or $N=0\oplus F$. In the latter case, $N$ is a direct summand of $M$.

\noindent\textit{Case 3.} $\pi_T(N)=T$ and $\pi_F(N)=F$. \\
We show that \(N\) is essential in \(M\). Pick elements
\[
v=(u,x)\in N \quad\text{and}\quad w=(t,b)\in N,
\]
where \(x\) generates \(F\) and \(t\in T\) has order \(p\); such elements exist by surjectivity of the projections.  
Let $\mathrm{o}(u)=p^k$. Then
\[
p^{k-1}v=(p^{k-1}u,0)\in p^{k-1}M\cap N=p^{k-1}N,
\]
so \((s,0):=(p^{k-1}u,0)\in N\) with \(s\) of order \(p\). Hence \(N\) meets the \(T\)-socle nontrivially, and since the socle of \(T\) is one-dimensional over \(\mathbb{Z}/p\mathbb{Z}\), we have \(\soc (T)\oplus 0\leq N\).

Next, from \(v=(u,x)\in N\) and \((s,0)\in N\) we obtain \(v-(s,0)=(u-s,x)\in N\).  
Let $\ell\ge1$ be minimal with \(p^{\ell}(u-s)=0\); then
\[
p^{\ell-1}(v-(s,0))=(p^{\ell-1}(u-s),0)\in N,
\]
and the first component has order \(p\), so it lies in \(\soc (T)\leq N\). Subtracting this socle element from \(v-(s,0)\) yields an element of the form \((0,x')\in N\) with \(x'\ne0\) in \(F\).  
Since \(F\) is simple, \(x'\) generates \(F\), hence \(0\oplus F\leq N\).

Thus \(N\) contains both \(\soc(T)\oplus0\) and \(0\oplus F\), i.e. $\soc (M)\leq N$. Because \(M\) is a torsion artinian module, \(\soc(M)\) is essential in \(M\); hence \(N\) is essential in \(M\). Therefore, in Case 3, the RD-pure submodule \(N\) is essential in the direct summand \(M\) itself.

Combining the three cases, every RD-pure submodule $N\le M$ is essential in a direct summand of $M$, so $M$ is RD-pure extending.

Therefore, the module $M=T\oplus F$ is RD-pure extending but not pure extending, so the class of RD-pure extending modules strictly contains the class of pure extending modules. \qed
\end{example}

\begin{remark}
In Example \ref{RD-pure not pure}, the module \(M=\mathbb{Z}(p^{\infty})\oplus \mathbb{Z}/p\mathbb{Z}\) is not projective over \(\mathbb{Z}\) because it is torsion, whereas every projective \(\mathbb{Z}\)-module is free \cite{lam1999lectures}*{text below Example~2.8}, and hence torsion-free. This provides a concrete example of a module that is RD-pure extending but neither projective nor pure extending, illustrating the necessity of the projectivity assumption in Proposition \ref{free(proj.) RD-PE <=> PE}. \qed
\end{remark}

\bibliography{ref}

\end{document}